\documentclass[a4paper]{amsart}
\usepackage{amsfonts,amsmath,amssymb,amscd,amstext,amsbsy,latexsym}
\newtheorem{thm}{Theorem}[section]
\newtheorem{cor}[thm]{Corollary}
\newtheorem{lem}[thm]{Lemma}
\newtheorem{prop}[thm]{Proposition}
\theoremstyle{definition}

\theoremstyle{remark}
\newtheorem{rem}[thm]{Remark}

\numberwithin{equation}{section}

\begin{document}

\title{Maximal subrings of Certain Non-commutative Rings}%
\author{Alborz Azarang}%

\maketitle

\centerline{Department of Mathematics, Faculty of Mathematical Sciences and Computer,}
\centerline{ Shahid Chamran University
of Ahvaz, Ahvaz-Iran} \centerline{a${}_{-}$azarang@scu.ac.ir}
\centerline{ORCID ID: orcid.org/0000-0001-9598-2411}
\begin{abstract}
The existence of maximal subrings in certain non-commutative rings, especially in rings which are integral over their centers, are investigated. We prove that if a ring $T$ is integral over its center, then either $T$ has a maximal subring or $T/J(T)$ is a commutative Hilbert ring with $|Max(T)|\leq 2^{\aleph_0}$ and $|T/J(T)|\leq 2^{2^{\aleph_0}}$. We observe that if $T$ is an algebraic $K$-algebra over a field $K$, then either $T$ has a maximal subring or $U(T)$ is integral over the prime subring of $T$. If $T$ is a left Artinian ring which is integral over its center, then we prove that either $T$ has a maximal subring or $T$ is countable and is integral over its prime subring. We see that if $T$ is a left Noetherian ring which is integral over its center, then either $T$ has a maximal subring or $|T|\leq 2^{\aleph_0}$. We prove that if $T$ is a domain which is integral over its center $C$ and $J(C)=0$, then either $T$ has a maximal subring or $T$ is an integral domain. If $T$ is a reduced ring which is integral over its center and the center of $T$ is a Hilbert ring, then we show that either $T$ has a maximal subring or $T$ is commutative. We see that if a ring $T$ is integral over its center and $R$ is a subring of $T$ with $J(T)\cap R\subseteq J(R)$, then either $T$ has a maximal subring or $J(R)=J(T)\cap R$ and $U(R)=U(T)\cap R$. Finally, we prove that if $T$ is direct product of an infinite family of rings $\{T_i\}_{i\in I}$ and each $T_i$ is integral over its center, then $T$ has a maximal subrings.
\end{abstract}
\section{Introduction}
Let $T$ be an associative ring with identity $1\neq 0$. In this paper we study the existence of maximal subrings (which contains $1_T$) in $T$. In \cite{azkrf} and \cite{azkra}, the authors characterized fields and Artinian commutative rings which have maximal subrings, respectively. Moreover, in \cite{azsid, azn, azconch, azkra, azkrc,azkrmc}, the existence of maximal subrings in more commutative rings investigated and obtained some facts as follows (that we use them in this paper in sequel):

\begin{enumerate}
\item \cite[Proposition 2.4]{azkra} each Artinian ring with zero characteristic has a maximal subring. Moreover, if an Artinian ring $R$ has no maximal subring, then $R$ is integral over $\mathbb{Z}_n$, for some $n>1$, see \cite[Corollary 2.5]{azkra}.
\item \cite[Corollary 2.4]{azn} each uncountable Noetherian ring has a maximal subring.
\item \cite[Theorem 3.1]{azsid} each uncountable UFD has a maximal subring.
\item \cite[Theorem 3.17]{azkrc} each infinite direct product of commutative rings has a maximal subrings.
\item \cite[Corollary 3.5]{azkrmc} either a ring has a maximal subring or is a Hilbert ring (i.e., each prime ideal is an intersection of a family of maximal ideals).
\item \cite[Corollary 2.5]{azconch} either a ring $T$ has a maximal subring or the unit group of $T$ is integral over the prime subring of $T$ (equivalently, either a ring $T$ has a maximal subring or for each subring $R$ of $T$ we have $U(R)=U(T)\cap R$, see \cite[Theorem 2.4]{azconch}).
\item \cite[Corollary 3.7]{azkrmc} each reduced ring of cardinality $>2^{2^{\aleph_0}}$ has a maximal subring. Moreover, if $R$ is a reduced ring with $|R|=2^{2^{\aleph_0}}$, then either $R$ has a maximal subring or $|Max(R)|=2^{\aleph_0}$, see \cite[Theorem 2.23]{azconch}.
\item \cite[Theorem 2.24]{azconch} each one-dimensional integral domain of cardinality $2^{2^{\aleph_0}}$ has a maximal subring.
\item \cite[Example 3.19]{azkrmc} for each infinite cardinal number $\mathfrak{a}$, there exists a ring of cardinality $\mathfrak{a}$ which has no maximal subring (note that for finite cardinal number $n$, the ring $\mathbb{Z}_n$ has no maximal subring).
\end{enumerate}

For non-commutative rings, in \cite{laffey} and \cite{klein}, the authors proved that if a ring $T$ has a finite maximal subring, then $T$ is finite too, see also \cite{lee} for a generalization of this result. The existence of maximal subrings for non-commutative rings first studied in \cite{azq}, we will mention some of the needed results of \cite{azq} in the next section. It is interesting to know that each ring $R$ can be considered as a maximal subring of a larger ring $T$, see \cite[Theorem 3.7]{azq}. Recently, in \cite{azdiv} the existence and the structure of maximal subrings in division rings investigated. We also refer the interested reader to \cite{azcond, azid}, for more results about the conductor ideals and certain ideals of the ring extension $R\subseteq T$, where $R$ is a maximal subring of a ring $T$.\\

If $R$ is a maximal subring of a ring $T$, then the extension $R\subseteq T$ is called a minimal ring extension, see \cite{frd,dbsid,dbsc,adjex,abmin}, for minimal ring extension of commutative rings. In \cite{dorsy} the authors studied minimal ring extensions of non-commutative rings. In fact, the authors generalized the results of \cite{dbsid} for non-commutative rings and characterized exactly (central) minimal ring extensions of prime rings and therefore simple rings, see \cite[Theorem 5.1, Corollary 5.3 and Theorem 6.1]{dorsy}. Also, they proved that a field $K$ has a minimal ring extension of the form $\mathbb{M}_n(D)$ where $D$ is a centrally finite division ring and $n>1$, if and only if $K$ has a proper subfield of finite index, see \cite[Lemma 6.6]{dorsy}. It is also interesting to know that if $D$ is a division ring then by the results in \cite{azq}, we have the following chain of maximal subrings:

$$
R_1=\{\begin{pmatrix}
  a & 0 & 0 \\
  0 & a & 0 \\
  0 & 0 & a
\end{pmatrix}\ |\ a\in D\}<
R_2=\{\begin{pmatrix}
  a & 0 & 0 \\
  0 & a & 0 \\
  0 & 0 & b
\end{pmatrix}\ |\ a, b\in D\} <
R_3=\begin{pmatrix}
  D & 0 & 0 \\
  0 & D & 0 \\
  0 & 0 & D
\end{pmatrix}<$$
$$R_4=\begin{pmatrix}
  D & 0 & 0 \\
  0 & D & D \\
  0 & 0 & D
\end{pmatrix}<
R_5=\begin{pmatrix}
  D & 0 & D \\
  0 & D & D \\
  0 & 0 & D
\end{pmatrix}<
R_6=\begin{pmatrix}
  D & D & D \\
  0 & D & D \\
  0 & 0 & D
\end{pmatrix}<
R_7=\begin{pmatrix}
  D & D & D \\
  D & D & D \\
  0 & 0 & D
\end{pmatrix}<\mathbb{M}_3(D)
$$

In \cite{blair}, the author studied non-commutative rings which are integral over their centers, especially in right noetherian rings. The author proved certain properties for such rings similar to commutative rings, we use the result of \cite{blair} to generalize the previous results that mentioned about the existence of maximal subrings in commutative rings to non-commutative rings that are integral over their centers. The ring $T=\mathbb{M}_n(K)$, where $K$ is a field, is an example of a ring which is integral over its center; in fact each subring of $T$ which contain the center of $T$ is integral over its center too. By the previous fact and use the structure of example \cite[P. 40]{lam}, for each natural number $n$ let $T_n:=\mathbb{M}_{2^n}(K)$, where $K$ is a field, then $T:=\bigcup_{n=1}^\infty T_n$ is a ring which is integral over its center.\\

A brief outline of this paper is as follows. In Section 2, we give some needed facts from \cite{azq, azcond, azdiv} and also some new immediate facts. We prove that certain rings either have a maximal subring or are duo rings. We see that a Von Neumann regular ring either has a maximal subring or is a duo ring. If $R$ is a left V-ring, then either $R$ has a maximal subring or $R$ is a duo ring. We observe that, if $R$ is a principal ideal domain, then either $R$ has a maximal subring (of the form $xRx^{-1}\cap R$, for some atom $x\in R$) or $R$ is a duo (Ore) domain. Section 3, of this paper is about the existence of maximal subrings in non-commutative rings, especially when a ring is integral over its center. We show that if $T$ is a ring which is integral over its center $C(T)$, $A$ is a conch maximal subring of $C(T)$ (i.e., $A$ has an element $u$ such that $u^{-1}\in C(T)\setminus A$) and the integral closure of $A$ in $T$, say $B$, is a subring of $T$, then $T$ has a maximal subring (and in fact $T$ has a maximal subring $R$ such that $R\cap C(T)=A$ and $u^{-1}\notin R$, i.e., we can extend a conch maximal subring of $C(T)$ to $T$ in the integral extension $C(T)\subseteq T$). For contraction, we see that if $R\subseteq T$ is an extension of rings, $V$ is a maximal subring of $T$ with $x\in U(R)\setminus V$, $xV=Vx$ (in particular, if $V$ is a duo ring) and either $V$ is left integrally closed in $T$ or $x^{-1}\in V$, then $R$ has a maximal subring. We show that, if $T$ is a ring which is integral over its center and $J(T)=0$, then either $T$ has a maximal subring or $T$ is a commutative ring. In fact, we prove that if $T$ is a ring with $J(T)=0$, then either $T$ has a maximal subring or $C(T)$ is integrally closed in $T$. Moreover, if $T$ is a prime ring with $J(T)=0$, then either $T$ has a maximal subring or $C(T)$ is algebraically closed in $T$. We prove that, if $T$ is an algebraic $K$-algebra over a field $K$, then either $T$ has a maximal subring or $U(T)$ is algebraic over the prime subring of $T$.\\

In Section 4, we study the existence of maximal subrings in rings which are integral over their centers with some algebraic properties on $T$ or $C(T)$, especially when either $T$ or $C(T)$ is Artinian, Noetherian, Hilbert and also rings which are direct product of an infinite family of rings (which are integral over their centers). We observe that if $T$ is a left Noetherian $K$-algebra, where $K$ is a field and $T$ is algebraic over $K$, then either $T$ has a maximal subring or $T$ is a left artinian ring. In particular, if $T$ has no maximal subring, then $T/J(T)$ is countable and $K$ is an absolutely algebraic field. We see that if $T$ is a left artinian ring which is integral over its center, then either $T$ has a maximal subring or $T$ is countable and is integral over its prime subring. We prove that if $T$ is a left noetherian ring which is integral over its center, then either $T$ has a maximal subring or $|T|\leq 2^{\aleph_0}$. We show that if $T$ is a reduced ring which is integral over its center and $C(T)$ is a Hilbert ring, then either $T$ has a maximal subring or $T$ is a commutative (Hilbert) ring. We prove that if $T$ is a ring which is integral over its center and $R$ is a subring of $T$ with $J(T)\cap R\subseteq R$, then either $T$ has a maximal subring or $J(R)=J(T)\cap R$ and $U(R)=U(T)\cap R$. We see that if $T$ is a direct product of an infinite family of rings $\{T_i\}_{i\in I}$ and each $T_i$ is integral over its center (in particular, if $T$ is integral over its center), then $T$ has a maximal subrings.\\

All rings in this paper are unital with $1\neq 0$. All subrings, modules and homomorphisms are also unital. If $R\subsetneq T$ is a ring extension and there exists no other subring between $R$ and $T$, then $R$ is called a maximal subring of $T$, or the extension $R\subseteq T$ is called a minimal ring extension. If $T$ is a ring, $I$ is an ideal of $T$ and $M$ is a left (resp. right) $T$-module, then $Min_T(I)$, $Max_r(T)$, $Max_l(T)$, $Max(T)$ and $l.ann_T(M)$ (resp. $r.ann_T(M)$), denote the set of all minimal prime ideals of $I$ in $T$, the set of all maximal right ideals of $T$, the set of all maximal left ideals of $T$, the set of all maximal ideals of $T$ and the left annihilator of $M$ in $T$ (resp. the right annihilator of $M$ in $T$), respectively. We use $Min(T)$ for $Min_T(0)$. The set of all left (resp. right) primitive ideals of $T$ is denoted by $Prm_l(T)$ (resp. $Prm_r(T)$). The characteristic of a ring $T$ is denoted by $Char(T)$. If $T$ is a ring, then $dim(T)$ denotes the classical Krull dimension of $T$ (i.e., the supremum of the lengths of all chains of prime ideals of $T$). We use $v.dim_K(T)$ for vector dimension of a $K$-algebra $T$ over a field $K$. $J(T)$ denotes the Jacobson radical of a ring $T$. For a ring $T$, $Nil_*(T)$, $Nil^*(T)$, $N(T)$ and $U(T)$ are the lower nilradical, the upper nilradical of $T$, the set of all nilpotent elements of $T$ and the units group of $T$, respectively. A prime ideal $P$ of a ring $T$ is called strongly (resp. completely) prime if $T/P$ has no nonzero nil ideal (resp. $T/P$ is a domain). If $T$ is a ring, then $\mathbb{M}_n(T)$ denote the ring of all $n\times n$ square matrices over $T$. If $M$ is a left (resp. right) module over a ring $T$, then the ring of all $T$-module homomorphisms of $M$ is denoted by $End({}_TM)$ (resp. $End(M_T)$). If $A$ is a right ideal of a ring $T$, then $\mathbb{I}(A)$ denotes the idealizer of $A$ in $T$, that is the largest subring of $T$ which $A$ is a two-sided ideal in it, i.e., $\mathbb{I}(A)=\{t\in T\ |\ tA\subseteq A\}$, the similar notation is used when $A$ is a left ideal of $T$. If $X$ is a subset of a ring $T$, then $C_T(X)$ is the centralizer of $X$ in $T$, in particular, $C(T)=C_T(T)$ is the center of $T$. For a ring $T$, the subring $\{n.1_T\ |\ n\in\mathbb{Z}\}$ is called the prime subring of $T$. A ring $T$ is called left (resp. right) quasi duo if each maximal left (resp. right) ideal of $T$ is a two-sided ideal of $T$, see \cite{lamq}. $T$ is called quasi duo if $T$ is left and right quasi duo ring. A ring $T$ is called left (resp. right) duo if each left (resp. right) ideal of $T$ is two-sided. Similarly duo rings are defined. If $R$ is a subring of a ring $T$ and $t\in T$, then we say that $t$ is left (resp. right) integral over $R$, if $t$ is a root of a left (resp. right) monic polynomial of degree $n\geq 1$ over $R$. $R$ is called left (resp. right) integrally closed in $T$, if every left (resp. right) integral element of $T$ over $R$, belongs to $R$. $R$ is called integrally closed in $T$, whenever $R$ is left and right integrally closed in $T$. If $T$ is a ring and $R$ is a subring of the center of $T$, then $T$ is called an $R$-algebra. It is clear that, if $T$ is an $R$-algebra, then the definitions of right and left integral elements and also left or right integrally closed for the extension $R\subseteq T$ are coincided. For other notations and definitions we refer the reader to \cite{good,lam,lam2,rvn}.

\section{Preliminaries}
We remind the reader that, a proper subring $S$ of a ring $T$ is maximal if and only if for each $x\in T\setminus S$, we have $S[x]=T$. By the previous fact and a natural use of Zorn's lemma one can easily see that, if $R$ is a proper subring of a ring $T$ and there exists $\alpha\in T$ such that $R[\alpha]=T$, then $T$ has a maximal subring $S$ with $R\subseteq S$ and $\alpha\notin S$. In particular, if $R\subsetneq T$ is a ring extension which is finite as ring (module) over $R$, then $T$ has a maximal subring which contains $R$. Therefore, if $T$ is a ring which has a division subring $D\neq T$ and $T$ has finite left/right vector dimension over $D$, then $T$ has a maximal subring. In this section we review some results from \cite{azq, azcond, azdiv} and obtain some new results too. Some important rings have naturally maximal subrings as follows:

\begin{thm}\label{t1}
\begin{enumerate}
\item \cite[Theorem 4.1]{azq}. Let $T$ be a ring and $A$ be a maximal right/left ideal of $T$ which is not an ideal of $T$. Then the idealizer of $A$ is a maximal subring of $T$. In particular, either a ring has a maximal subring or is a quasi duo ring.
\item \cite[Proposition 4.2]{azq}. Let $R$ be a maximal subring of a ring $T$ which contains a maximal one-sided ideal $A$ of $T$ which is not an ideal of $T$. Then $R$ is the idealizer of $A$ in $T$.
\item \cite[Theorem 4.4]{azq}. Let $T$ be a ring which is not a division ring. If $T$ is left/right primitive ring (in particular, if $T$ is a simple ring), then $T$ has a maximal subring.
\item \cite[Theorem 3.1]{azq}. For each ring $R$ and $n>1$, the matrix ring $\mathbb{M}_n(R)$ has a maximal subring.
\end{enumerate}
\end{thm}
In particular, if a ring $T$ has no maximal subring, then $T$ is right primitive if and only if $T$ is left primitive; moreover, $\mathcal{M}:=Max(T)=Max_l(T)=Max_r(T)=Prm_l(T)=Prm_r(T)$, for each $M\in\mathcal{M}$, the ring $\frac{T}{M}$ is a division ring and $J(T)=\bigcap_{M\in Max(T)}M$. Therefore $T/J(T)$ embeds in $\prod_{M\in\mathcal{M}} T/M$. In particular, $T/J(T)$ is a reduced ring. Hence we have the following immediate result.

\begin{cor}\label{t2}
Let $T$ be a ring, then either $T$ has a maximal subring or $N(T)\subseteq J(T)$.
\end{cor}

We remind that if $T$ is a commutative ring, then either $T$ has a maximal subring or $T$ is a Hilbert ring, see \cite[Corollary 3.5]{azkrmc}. In particular, if $T$ has no maximal subring, then $N(T)=J(T)$.

\begin{cor}\label{t3}
Let $T$ be a ring, then either $T$ has a maximal subring or for each subring $R$ of $T$ and maximal left/right ideal $M$ of $T$, the ideal $R\cap M$ is a completely prime ideal of $R$.
\end{cor}
\begin{proof}
If $T$ has no maximal subring, then for each maximal left/right ideal of $T$, the ring $T/M$ is a division ring. Since $(R+M)/M$ is a subring of $T/M$, we deduce that $(R+M)/M$ is a domain. Now by the ring isomorphism $R/(R\cap M)\cong (R+M)/M$, the conclude that $R/(R\cap M)$ is a domain, i.e., $R\cap M$ is a completely prime ideal of $R$.
\end{proof}

One can easily see that, if $T$ is a quasi duo ring and $x_1,\ldots,x_n\in T$, then $Tx_1+\cdots+Tx_n=T$ if and only if $x_1T+\cdots+x_n T=T$ and the converse also holds. In fact, if $T$ is a quasi duo ring and $I$ is a proper left (resp. right) ideal of $T$, then $IT$ (resp. $TI$) is a proper ideal of $T$. Consequently, each quasi duo ring is a Deddkind finite ring (i.e., whenever $ab=1$ then $ba=1$). Now we have the following result.

\begin{cor}\label{t4}
Let $F=\bigoplus_{i\in I} R$ be a free left $R$-module over a ring $R$ and $|I|\geq 2$. Then $E=End_R(F)$ has a maximal subring.
\end{cor}
\begin{proof}
If $|I|=n$ is finite, then clearly $E\cong \mathbb{M}_n(R)$ and therefore $E$ we are done by $(4)$ of Theorem \ref{t1} (also note that in fact in this case $E$ is not a quasi duo ring). Thus assume that $I$ is infinite. We may assume that $I=\mathbb{N}\cup J$ where $J\cap \mathbb{N}=\emptyset$ and $\{e_i\ |\ i\in I\}$ be the natural basis for $F$. Now define $A, B \in E$ as follows. For each $i\in \mathbb{N}$, $A(e_i)=e_{i+1}$ and otherwise $A(e_i)=e_i$, for $i\in J$, and $B(e_1)=0$, $B(e_i)=e_{i-1}$ for $i\geq 2$ and $B(e_i)=e_i$ for $i\in J$. Then one can easily see that $BA=1_E$ but $AB\neq 1_E$. Thus $E$ is not a Deddkinf finite ring and therefore is not a quasi duo ring. Hence $E$ has a maximal subring.
\end{proof}

Note that if $R$ is a ring, then $R$ has a maximal subring if and only if $E=End({}_RR)\cong R^{op}$ (resp. $End(R_R)\cong R$) has a maximal subring. Since any left/right vector space over a division ring is free we have the following immediate result.

\begin{cor}\label{t5}
Let $D$ be a division ring and $V$ be a left/right vector space over $D$ of vector dimension $>1$. Then $End_D(V)$ has a maximal subring.
\end{cor}

In what following, we see that certain rings which have no maximal subring are duo.

\begin{cor}\label{t6}
Let $R$ be a von Neumann regular ring. Then either $R$ has a maximal subring or $R$ is a strongly von Neumann regular (reduced duo) ring.
\end{cor}
\begin{proof}
First note that $J(R)=0$, for $R$ is a von Neumann regular ring. Hence if $R$ has no maximal subring, then by Corollary \ref{t2}, $R$ is a reduced ring. Therefore $R$ is a strongly von Neumann regular ring and therefore is a duo ring.
\end{proof}

If in the previous corollary we assume that the center of $R$ is a field, then we have a better conclusion as follow.

\begin{cor}
Let $T$ be a von Neumann regular ring and $C(T)=F$ is a field. Then either $T$ has a maximal subring or $T$ is a division ring.
\end{cor}
\begin{proof}
Assume that $T$ has no maximal subring, thus by Corollary \ref{t6}, we infer that $T$ is a reduced ring which immediately implies that each idempotent element of $T$ is central. Now assume that $0\neq x\in T$, thus by our assumption there exists $y\in T$ such that $x=xyx$. It is easy to see that $xy$ and $yx$ are nonzero idempotent and therefore $xy$ and $yx$ are central. Therefore $xy=1=yx$, for $C(T)=F$ is a field.
\end{proof}

We remind that a ring $T$ is called a left $V$-ring, if each simple left $T$-modules are injective, see \cite[Section 3H]{lam2}.

\begin{prop}\label{t7}
Let $T$ be a left V-ring. Then either $T$ has a maximal subring or $T$ is a duo ring. In particular, if $T$ is a left $V$-ring which has no maximal subring, then $T$ is a right $V$-ring.
\end{prop}
\begin{proof}
Since $T$ is a left V-ring, then by \cite[Theorem 3.75]{lam2}, we infer that $J(T)=0$ and also for each ideal $I$ of $T$, $T/I$ is a left V-ring too. Now assume that $T$ has no maximal subring (or $T$ is a quasi duo ring), then by Corollary \ref{t6}, we deduce that $T$ is a reduced ring. Now we claim that $T$ is von Neumann regular ring. To see this, note that by \cite[Theorem 1.21]{vnrg}, it suffices to show that for each completely prime ideal $P$ of $T$, the ring $R/P$ is von Neumann regular. Hence assume that $P$ is a completely prime ideal of $T$ (i.e., $T/P$ is a domain), by the first part of the proof $T/P$ is a left V-ring, and therefore by \cite[P. 335, Ex. 23]{rvn}, we deduce that $T/P$ is a simple ring. Since $T$ has no maximal subring (or is a quasi duo ring) we conclude that $T/P$ is a division ring, by $(3)$ of Theorem \ref{t1}. Thus $T/P$ is a von Neumann regular and therefore $T$ is a von Neumann regular ring. Hence by Corollary \ref{t6}, $T$ is a duo ring. The final part is evident now, for $T$ is a duo ring and use \cite[Theorem 3.75]{lam2}.
\end{proof}

For the next result we need the following lemma.

\begin{lem}\label{t8}
Let $R\subseteq T$ be an extension of rings and $x\in (R\cap U(T))\setminus U(R)$. Then $\mathbb{I}_R(xR)=xRx^{-1}\cap R$.
\end{lem}
\begin{proof}
It is not hard to see that $xRx^{-1}$ is a subring of $T$ (which is isomorphic to $R$ as ring). Thus $xRx^{-1}\cap R$ is a subring of $R$ and it is easy to see that $xR\subseteq xRx^{-1}$ and $xR$ is an ideal of $xRx^{-1}\cap R$. Thus $xRx^{-1}\cap R\subseteq \mathbb{I}_R(xR)$. Conversely, assume that $r\in\mathbb{I}_R(xR)$. Thus $r\in R$ and $rxR\subseteq xR$. Hence $rx\in xR$ and therefore $r\in xRx^{-1}$. Thus $r\in R\cap xRx^{-1}$ and we are done.
\end{proof}

For the definition of atom and (left/right) principal ideal ring,  we refer the reader to \cite[Sec 1.3]{cohnfir}.

\begin{thm}\label{t9}
Let $R$ be a prime principal ideal ring. Then either $R$ has a maximal subring or $R$ is a duo principal ideal domain. Moreover, if $R$ is a principal ideal domain which is not a duo domain, then $R$ has a maximal subring of the form $xRx^{-1}\cap R$ or $x^{-1}Rx\cap R$, for some atom $x\in R$.
\end{thm}
\begin{proof}
First note that since $R$ is a prime principal ideal ring, by Goldie's Principal Ideal Ring Theorem, see \cite[P. 65]{faith}, $R\cong\mathbb{M}_n(S)$, where $S$ is a principal ideal (Ore) domain. Hence if $n>1$, then $R$ has a maximal subring by $(4)$ of Theorem \ref{t1}, and if $n=1$, then $R$ is a principal ideal (Ore) domain. Hence assume that $R$ is a principal ideal domain, we prove that either $R$ has a maximal subring or $R$ is a duo domain. To see this, assume that $R$ has no maximal subring, then $R$ is a quasi duo ring by $(1)$ of Theorem \ref{t1}. Hence $Max(R)=Max_l(R)=Max_r(R)$ and since $R$ is a principal ideal ring we immediately conclude that $Max(R)=\{Rp\ |\ p\in\mathcal{A}\}$, where $\mathcal{A}$ is the set of all atoms of $R$ (up to associate). Now by \cite[Theorem 1.3.5]{cohnfir}, we deduce that $R$ is a unique factorization domain. Hence if $a$ is a nonzero nonunit element of $R$, then $a=u_1p_1u_2\cdots u_np_nu_{n+1}$, where $u_i\in U(R)$ and $p_i\in\mathcal{A}$, for some natural number $n$. Since $R$ is a quasi duo ring we infer that $Rp=pR$, for each $p\in\mathcal{A}$, which immediately implies that $Ra=Ra$ and therefore $R$ is a duo ring and hence we are done for the first part. For the final part, if $R$ is not a duo domain, then we deduce that $R$ is not a quasi duo domain. We may assume that $R$ is not a right quasi domain. Thus $R$ has a maximal right ideal $M$ which is not an ideal of $R$. Since $R$ is a (right) principal ideal ring, we conclude that $M=pR$ and therefore we are done by $(1)$ of Theorem \ref{t1}, the previous lemma and the fact that $R$ is an Ore domain (note, each (right) principal ideal domain is an Ore domain).
\end{proof}

We need the following three results about the existence of maximal subrings in division rings from \cite{azdiv}.

\begin{cor}\label{t10}
Let $D$ be a division ring with center $F$. Then either $D$ has a maximal subring or each $\alpha\in D\setminus F$ is not algebraic over $F$.
\end{cor}

In other words, if a division ring $D$ with center $F$ has no maximal subring, then $F$ is algebraically closed in $D$. We will generalized this result to certain rings (such as prime rings) in the next section.

\begin{cor}\label{t11}
Let $D$ be a division ring with center $F$ which is algebraic over its center. Then $D$ has no maximal subrings if and only if $D=F$ is a field without maximal subrings.
\end{cor}

For the characterization of fields which have no maximal subrings we refer the reader to \cite{azkrf}.

\begin{cor}\label{t12}
Let $D$ be a non-commutative division ring with center $F$. Then either $D$ has a maximal subring or $dim_F(D)\geq |F|$.
\end{cor}

Finally in this section we have some observation about the conductor ideals of a maximal subring of a ring. Let $R\subseteq T$ be a ring extension, then we have three type of conductors: $(R:T):=\{x\in T\ |\ TxT\subseteq R\}$, $(R:T)_l:=\{x\in T\ |\ Tx\subseteq R\}$, $(R:T)_r:=\{x\in T\ |\ xT\subseteq R\}$. In other words, $(R:T)$ is the largest common ideal between $R$ and $T$, $(R:T)_l$ (resp. $(R:T)_r$)  is the largest common left (resp. right) ideal between $R$ and $T$. It is clear that $(R:T)_l=r.ann_R((\frac{T}{R}))$ and therefore $(R:T)_l$ is an ideal of $R$. Similarly, $(R:T)_r=l.ann_R(\frac{T}{R})$ is an ideal of $R$. Finally note that $(R:T)_l(R:T)_r\subseteq (R:T)\subseteq (R:T)_l\cap (R:T)_r$. If $R$ is a maximal subring of a commutative ring $T$, then $(R:T)=(R:T)_l=(R:T)_r$ is a prime ideal of $R$, see \cite{frd}. We have the following for non-commutative ring from \cite[Lemma 2.1]{azcond}.

\begin{thm}\label{t13}
Let $R$ be a maximal subring of a ring $T$. Then $(R:T)_l$ and $(R:T)_r$ are prime ideals of $R$.
\end{thm}

\section{Maximal subring of rings integral over their centers}
In this section we investigate about the existence of maximal subring in rings which are integral over its centers, see \cite{blair}. Especially, algebraic $K$-algebras where $K$ is a field. We also obtain some observation about rings which are $J$-semisimple or the Jacobson radicals of them are nil. At first of this section we give two result about the existence of a maximal subring in a ring extension $R\subsetneq T$, i.e., we want to study whenever $R$ has a maximal subring, then under certain condition $T$ has a maximal subring and vice versa. We begin by the following which is a generalization \cite[Proposition 2.1]{azkrc} for commutative ring extension. The assumption that the integral closure of a central subring is a subring, in the next result, is similar to the extension of a central valuation subring $V$ of a division ring $D$ which is finite over its center, to a valuation for $D$, see \cite[Theorem 8.12]{marbo}. We remind the reader that, a subring $R$ of a ring $T$ is called a conch maximal subring, if there exists a unit $u\in T$ such that $u\in R$ but $u^{-1}\notin R$ and $R$ is maximal respect to this property. Now the following is in order.

\begin{thm}\label{t14}
Let $T$ be a ring which is integral over its center $C(T)$. If $A$ is a conch maximal subring of $C(T)$ and the integral closure of $A$ in $T$, say $B$, is a subring of $T$, then $T$ has a maximal subring. In fact, $T$ has a maximal subring $R\supseteq B$ and $R\cap C(T)=A$.
\end{thm}
\begin{proof}
Since $A$ is a conch maximal subring of $T$, we infer that there exists a unit $\alpha\in C:=C(T)$ such that $\alpha\in A$ and $\alpha^{-1}\notin A$. Hence $C=A[\alpha^{-1}]$, for $A$ is a maximal subring of $C$. Therefore each element of $C$ is of the form $a\alpha^{-k}$ for some $a\in A$ and $k\geq 0$ (note, $\alpha\in C$). It is clear that $\alpha^{-1}\notin B$, for $A$ is integrally closed in $C$ and $\alpha^{-1}\in C\setminus A$. We prove that $B[\alpha^{-1}]=T$. Assume that $x\in T$, since $x$ is integral over $C$, we conclude that there exist $n$ and $c_0,\ldots, c_{n-1}$ in $C$ such that $x^{n}+c_{n-1}x^{n-1}+\cdots+c_1x+c_0=0$. Since $C=A[\alpha^{-1}]$, we deduce that there exist $a_i\in A$ and $m\geq 0$ such that $c_i=a_i\alpha^{-m}$. Therefore $x^{n}+a_{n-1}\alpha^{-m}x^{n-1}+\cdots+a_1\alpha^{-m}+a_0\alpha^{-m}=0$. Multiplying this equation by $\alpha^{mn}$ (note, $\alpha\in C$), we obtain that $(\alpha^{m}x)^n+a_{n-1}(\alpha^m x)^{n-1}+a_{n-2}\alpha^{m}(\alpha^mx)^{n-2}+\cdots+a_1\alpha^{m(n-2)}(\alpha^mx)+a_0\alpha^{m(n-1)}=0$. Therefore $\alpha^mx\in B$. Hence $x\in B[\alpha^{-1}]$, which shows that $B[\alpha^{-1}]=T$. Thus by a natural use of Zorn's Lemma, $T$ has a maximal subring $R$ such that $B\subseteq R$ and $\alpha^{-1}\notin R$. Finally, note that $R\cap C$ is a subring of $C$ which contains $A$ but not $\alpha^{-1}$. Thus $R\cap C=A$, by maximality of $A$ and we are done.
\end{proof}

\begin{cor}\label{t15}
Let $T$ be a simple ring which is integral over its center. If the center of $T$ has a non-field maximal subring $A$ which the integral closure of $A$ in $T$ is a subring of $T$, then $T$ has a maximal subring.
\end{cor}
\begin{proof}
It is clear that $C(T)$ is a field, for $T$ is a simple ring. Since $A$ is a maximal subring of $C(T)$ which is not a field, we infer that there exists $u\in A$ such that $u^{-1}\notin A$. Thus $A$ is a conch maximal subring of $C(T)$ and hence we are done by Theorem \ref{t14}.
\end{proof}

In the next result we see that under certain conditions the existence of a maximal subring in a ring $T$ inherits to some subrings of $T$. This result is a generalization of \cite[Theorem 2.19]{azkrmc}.

\begin{prop}\label{t16}
Let $R\subseteq T$ be an extension of rings. Assume that $V$ is a maximal subring of $T$, $x\in U(R)\setminus V$ and $xV=Vx$ (in particular, if $V$ is a duo ring). If $x^{-1}\in V$ (in particular, if $V$ is left integrally closed in $T$), then $R$ has a maximal subring.
\end{prop}
\begin{proof}
First we prove that if $V$ is integrally closed in $T$, then $x^{-1}\in V$. To see this, since $x\notin V$ and $V$ is a maximal subring of $T$, we conclude that $T=V[x]$. From $xV=Vx$, we deduce that $T=V[x]=\{v_0+v_1x+\cdots+v_nx^n\ |\ n\in\mathbb{N}\cup\{0\},\ v_i\in V, \ 0\leq i\leq n\}$. Hence $x^{-1}=v_0+v_1x+\cdots+v_nx^n$, for some $n$ and $v_i\in V$. Multiplying from right by $x^{-n}$, we obtain that $x^{-1}$ is left integral over $V$ and therefore $x^{-1}\in V$. Thus we may assume that $R$ has a unit $u$ such that $u^{-1}\notin V$, $u\in V$ and $uV=Vu$. Hence by maximality of $V$, we have $V[u^{-1}]=T$. From $uV=Vu$ we conclude that $u^{-1}V=Vu^{-1}$, and therefore $T=V[u^{-1}]=V+Vu^{-1}+Vu^{-2}+\cdots=\{v_0+v_1u^{-1}+\cdots+v_n(u^{-1})^n\ |\ n\in\mathbb{N}\cup\{0\}$. Since $u\in V$, we deduce that $T=\{vu^{-n}\ |\ v\in V, n\geq 0\}$. Now we claim that $R\cap V$ is a proper subring of $R$ and $(R\cap V)[u^{-1}]=R$. It is clear that $R\cap V$ is a subring of $R$. Since $u\in U(R)$ and $u^{-1}\notin V$, we infer that $R\cap V$ is a proper subring of $R$. Let $r\in R$, then $r\in T$ and therefore there exist $v\in V$ and $n\geq 0$ such that $r=vu^{-n}$. Thus $v=ru^n\in R\cap V$. Therefore $r\in (R\cap V)[u^{-1}]$. Thus $(R\cap V)[u^{-1}]=R$, and since $R\cap V$ is a proper subring of $R$, by a natural use of Zorn's Lemma we conclude that $R$ has a maximal subring.
\end{proof}

\begin{rem}\label{t17}
Let $D$ be a division ring and $V$ be a (proper) DVR (see \cite[Section 3.5]{hazw}) for $D$. Then $V$ is a maximal subring of $D$. To see this we must show that for each $x\in D\setminus R$, we have $V[x]=D$. Since $V$ is a DVR for $D$, we infer that there exists a nonzero non-unit $\pi\in V$ such that each element $x$ of $D$ is of the form $x=u\pi^{n}=\pi^nv$, where $n\in\mathbb{Z}$ and $u,v\in U(V)$. This immediately implies that $V[\pi^{-1}]=D$. Since $x\in D\setminus V$, we conclude that $x=u\pi^n$, where $n<0$ and $u\in U(V)$. From $\pi\in V$ and $u\in U(V)$ we conclude that $T=V[\pi^{-1}]\subseteq V[\pi^n]=V[x]$ and therefore $V[x]=D$. Thus $V$ is a maximal subring of $D$.
\end{rem}

\begin{cor}\label{t18}
Let $R$ be a left/right Ore domain with division ring of quotients $D$. Let $V$ be a (proper) DVR for $D$ such that $U(R)\nsubseteq V$, then $R$ has a maximal subring.
\end{cor}
\begin{proof}
Let $x\in U(R)\setminus V$, since $V$ is integrally closed in $D$ (or note that one can easily see that a DVR is total valuation, i.e., for each $x\in D$, either $x\in V$ or $x^{-1}\in V$) we conclude that $x^{-1}\in V$. By \cite[Proposition 3.5.3]{hazw}, $V$ is a duo ring and therefore $xV=Vx$, hence we are done by Proposition \ref{t16} and Remark \ref{t17}.
\end{proof}

Now we have the following main result.

\begin{thm}\label{t19}
Let $T$ be a ring which is integral over its center and $J(T)=0$. Then either $T$ has a maximal subring or $T$ is a commutative ring.
\end{thm}
\begin{proof}
Assume that $T$ has no maximal subring, then by $(1)$ of Theorem \ref{t1}, we infer that $T$ is a quasi duo ring, and by the comments after Theorem \ref{t1}, we have $0=J(T)=\bigcap_{M\in Max(T)}M$; and for each maximal ideal $M$ of $T$, the ring $T/M$ is a division ring. Assume that $\alpha$ is a non-central element of $T$ (which is integral over the center of $T$). Hence there exists an element $\beta\in T$ such that $\alpha\beta-\beta\alpha\neq 0$. Since $J(T)=0$, we conclude that there exists a maximal ideal $M$ of $T$ such that $\alpha\beta-\beta\alpha\notin M$. Therefore $\alpha+M$ is not a central element of the division ring $T/M$. Clearly $\alpha+M$ remains integral over the center of $T/M$ and thus by Corollary \ref{t10}, we deduce that $T/M$ has a maximal subring. Thus $T$ has a maximal subring too which is absurd. Thus each element of $T$ (which is integral over the center of $T$) is central and hence $T$ is a commutative ring.
\end{proof}

Combining the previous result by some facts that we mentioned in the introduction of this paper, we obtain the next observation.

\begin{prop}\label{t20}
Let $T$ be a ring which is integral over its center. Then either $T$ has a maximal subring or $T$ is a quasi duo ring with $Max(T)=Max_r(T)=Max_l(T)$ that satisfies in the following conditions:
\begin{enumerate}
\item $T/J(T)$ is a commutative reduced Hilbert ring and $|T/J(T)|\leq 2^{2^{\aleph_0}}$.
\item $|Max(T)|\leq 2^{\aleph_0}$, and for each maximal ideal $M$ of $T$, $T/M$ is an absolutely algebraic field, in particular $|T/M|\leq {\aleph_0}$.
\item For any two distinct maximal ideal $M$ and $N$ of $T$, the fields $T/M$ and $T/N$ are not isomorphic.
\item If $J(T)=Nil_*(T)$, then for each prime ideal $P$ of $T$, the ring $T/P$ is an integral domain with $J(T/P)=0$, i.e., $P$ is an intersection of a family of maximal ideals of $T$.
\end{enumerate}
\end{prop}
\begin{proof}
Assume that $T$ has no maximal subring, thus by $(1)$ of Theorem \ref{t1}, we conclude that $T$ is a quasi duo ring. Hence $Max(T)=Max_l(T)=Max_r(T)$. Clearly, the ring $T/J(T)$ is integral over the central subring $(C(T)+J(T))/J(T)$ and therefore is integral over its center. Since $T$ has no maximal subring we conclude that $T/J(T)$ has no maximal subring and therefore by Theorem \ref{t19}, we deduce that $T/J(T)$ is a commutative ring. Since for each $M\in Max(T)$ we have $J(T)\subseteq M$, by \cite[Proposition 2.6]{azkrf}, we deduce that $|Max(T)|\leq 2^{\aleph_0}$ and by \cite[Corollary 1.4]{azkrf}, $T/M$ is an absolutely algebraic field and therefore $|T/M|\leq {\aleph_0}$. Also note that $T/J(T)$ embeds in $\prod_{M\in Max(T)} T/M$ and therefore $T/J(T)$ is a reduced ring with $|T/J(T)|\leq 2^{2^{\aleph_0}}$. By \cite[Corollary 3.5]{azkrmc}, we conclude that $T/J(T)$ is a Hilbert ring. Thus $(1)$ and $(2)$ hold. $(3)$ is an immediate consequence of \cite[$(1)$ of Corollary 2.4]{azq}. Finally for $(4)$, since $Nil_*(T)$ is the intersection of all prime ideals of $T$, by the assumption $J(T)=Nil_*(T)$, we conclude that $J(R)$ is contained in each prime ideal of $T$. Now since $T/J(T)$ is a commutative Hilbert ring, we conclude that for each prime ideal $P$ of $T$, $T/P$ is an integral domain with $J(T/P)=0$.
\end{proof}

By the above results we have several consequences as follows.

\begin{cor}\label{t21}
Let $T$ be a ring which is integral over its center. If $K$ is a field which is a subring of $T$ and $K$ is not an absolutely algebraic field, then $T$ has a maximal subring.
\end{cor}
\begin{proof}
If $T$ has no maximal subring, then by $(2)$ of Proposition \ref{t20}, $T/M$ is an absolutely algebraic field for each maximal ideal $M$ of $T$. Clearly $R/M$ contains a copy of $K$, which is not possible, for subfields of absolutely algebraic fields are absolutely algebraic. Hence we conclude that $T$ has a maximal subring.
\end{proof}

In particular, if $T$ is a ring which is integral over its center and $T$ contains a field $K$ such that either $K$ is uncountable or $Char(K)=0$, then $T$ has a maximal subring.
Moreover, if $T$ is a $K$-algebra over a field $K$ and $T=K[\{\alpha_i\ |\ i\in I\}]\neq K$, where $|I|<|K|$, then $T$ has a maximal subring. To see this, we have two cases. If $I$ is finite, then by a natural use of Zorn's Lemma one can easily see that $T$ has a maximal subring, otherwise assume that $I$ is infinite. Thus $K$ is uncountable, for $|I|<|K|$ and therefore by the latter fact we deduce that $T$ has a maximal subring.

\begin{cor}\label{t21a}
Let $T$ be a ring which is integral over its center. If $K$ is an infinite field which is a subring of $T$, then either $T$ has a maximal subring or there exists at most one maximal ideal $M$ of $T$ such that $v.dim({}_K(T/M))\neq |K|$ ($v.dim((T/M)_K)\neq |K|$). Moreover, if $v.dim({}_K(T/M))\neq |K|$ ($v.dim((T/M)_K)\neq |K|$), for a maximal ideal $M$ of $T$, then $T/M\cong K$.
\end{cor}
\begin{proof}
Assume that $T$ has no maximal subring, thus by previous corollary we deduce that $K$ is an absolutely algebraic field and therefore is countable. By $(2)$ of Proposition \ref{t20}, for each maximal ideal $M$ of $T$, $T/M$ is an absolutely algebraic field which clearly contains a copy of $K$. Since $K$ is infinite, we conclude that $v.dim({}_K(T/M))\leq |K|$ (note, otherwise $T/M$ is uncountable which is impossible). Now if the equality does not hold, we have $v.dim({}_K(T/M))$ is finite and since $T/M$ has no maximal subring, we deduce that $v.dim({}_K(T/M))=1$, (note, if $1<v.dim({}_K(T/M))=n$ then $K\subsetneq T/M$ is a finite ring extension and therefore $T/M$ has a maximal subring, hence $T$ has a maximal subring which is absurd), thus $T/M\cong K$. Finally note that if such $M$ exists, then $M$ is unique by $(3)$ of Proposition \ref{t20}.
\end{proof}

Also note that if $T$ is a $K$-algebra over a field $K$ and $v.dim_K(T)<|K|$, then $T$ has no maximal subring if and only if $T=K$ is an absolutely algebraic field without maximal subring. To see this, assume that $T$ has no maximal subring. Note that if $K$ is an uncountable field, then by Corollary \ref{t21}, we deduce that $T$ has a maximal subring which is absurd. Thus we conclude that $K$ is countable (and by Corollary \ref{t2}, is an absolutely algebraic field). Hence $v.dim_K(T)$ is finite. If $1<v.dim_K(T)$, then we observe that $K\subsetneq T$ is a finite ring extension and therefore $T$ has a maximal subring which is absurd. Thus $v.dim_K(T)=1$, i.e., $T=K$ is an absolutely algebraic field without maximal subring. The converse is evident.

\begin{cor}\label{t22}
Let $T$ be a ring which is integral over its center. Assume that $R$ is a subring of $T$ and there exists a maximal (left/right) ideal $M$ of $T$ such that $R\cap M$ is not maximal ideal of $R$ (in fact, $R/(R\cap M)$ is not an absolutely algebraic field), then $T$ has a maximal subring. In particular, if $T$ is a zero-dimensional ring which is integral over its prime subring and $R$ is a subring of $T$ which a non-field domain, then $T$ has a maximal subring.
\end{cor}
\begin{proof}
Assume that $T$ has no maximal subring, thus by $(2)$ of proposition \ref{t20}, $T/M$ is an absolutely algebraic field. Since $(R+M)/M$ is a subring of $T/M$, we conclude that $(R+M)/M$ is a field too. Thus $R/(R\cap M)$ is a field, i.e., $R\cap M$ is a maximal ideal of $R$ which is impossible (in fact since subrings of absolutely algebraic fields are absolutely algebraic, $R/(R\cap M)$ is an absolutely algebraic field). Hence $T$ has a maximal subring. For the second part, note that $R\setminus\{0\}$ is a multiplicatively closed set in $T$, thus $T$ has a prime ideal $Q$ such that $Q\cap R=0$. Since $T$ is a zero-dimensional ring, we infer that $Q$ is a maximal ideal of $T$, hence we are done by the first part.
\end{proof}

In \cite[Corollary 3.5]{azkrmc}, the authors proved that a commutative ring $R$ either has a maximal subring or is a Hilbert ring. We remind that a commutative ring $R$ is a Hilbert ring if and only if for each proper ideal $I$ of $R$ we have $N(R/I)=J(R/I)$. For algebraic $K$-algebra we have the following same result.

\begin{cor}\label{t23}
Let $T$ be an algebraic $K$-algebra, where $K$ is a field. Then either $T$ has a maximal subring or $(1)-(4)$ of Proposition \ref{t20}, hold for $T$. Moreover, if $T$ has no maximal subring, then $K$ is an absolutely algebraic field and for each proper ideal $I$ of $T$, $N(T/I)=J(T/I)$. In particular, if in addition $T$ is a reduced ring, then either $T$ has a maximal subring or $T$ is a commutative Hilbert ring.
\end{cor}
\begin{proof}
For the first part, since $T$ is algebraic over $K$ and $K$ is a central subfield of the center of $T$, we deduce that $T$ is integral over its center. Thus by Proposition \ref{t20}, either $T$ has a maximal subring or $T$ satisfies in items $(1)-(4)$ of Proposition \ref{t20}. Now assume that $T$ has no maximal subring, then by Corollary \ref{t21}, $K$ is an absolutely algebraic field. Since $T$ is an algebraic $K$-algebra, we conclude that for each proper ideal $I$ of $T$, the ring $T/I$ is an algebraic $K$-algebra. Therefore by \cite[Theorem 4.20]{lam}, we infer that $J(T/I)$ is a nil ideal of $T/I$, i.e., $J(T/I)\subseteq N(T/I)$. From the fact that $T$ has no maximal subring, we deduce that $T/I$ has no maximal subring and therefore by Corollary \ref{t2}, $N(T/I)\subseteq J(T/I)$. Thus $N(T/I)=J(T/I)$. The final part is evident, for in this case $J(T)=N(T)=0$.
\end{proof}

In fact, in the previous corollary, if $T$ is a reduced algebraic $K$-algebra over a field $K$, then either $T$ has a maximal subring or is a von Neumann regular commutative ring with prime characteristic $p$ and $T$ is integral over $\mathbb{Z}_p$. To see this, assume that $T$ has no maximal subring, then $K$ is an absolutely algebraic field and therefore $K$ and $T$ has prime characteristic. Since $T$ is commutative and algebraic over $K$, we infer that $T$ is a zero-dimensional ring. Thus $T$ is a reduced zero-dimensional ring, hence is a von Neumann regular ring. By \cite[Corollary 3.14]{azkrc}, we deduce that $T$ is integral over $\mathbb{Z}_p$.\\

We remind that by \cite[Corollary 2.5]{azconch}, if $T$ is a commutative ring, then either $T$ has a maximal subring or $U(T)$ is integral over the prime subring of $T$. Now we have the following result.

\begin{prop}\label{t24a1}
Let $T$ be a ring which is integral over its center. If $J(T)$ is nil, then either $T$ has a maximal subring or $U(T)$ is integral over the prime subring of $T$. In particular, if $T$ is an algebraic algebra over a field $K$, then either $T$ has a maximal subring or $U(T)$ is integral over the prime subring of $T$.
\end{prop}
\begin{proof}
Assume that $T$ is a ring which is integral over its center, say $C$, and $J(T)$ is nil. We prove that either $T$ has a maximal subring or $U(T)$ is integral over the prime subring of $T$, say $\mathbf{Z}$. Suppose that $T$ has no maximal subring. Thus by $(1)$ of Proposition \ref{t20}, $T/J(T)$ is a commutative ring which has no maximal subring, for $T$ has no maximal subring. Thus $U(T/J(T))$ is integral over the prime subring of $T/J(T)$, by \cite[Corollary 2.5]{azconch}. Now let $x\in U(T)$, clearly $x+J(T)$ is a unit in $T/J(T)$. Thus we immediately conclude that there exist $n\geq 1$ and $a_0,a_1,\ldots,a_{n-1}\in\mathbf{Z}$ such that $x^n+a_{n-1}x^{n-1}+\cdots+a_1x+a_0\in J(T)$. Since $J(T)$ is nil, we deduce that there exists $m\geq 1$, such that $(x^n+a_{n-1}x^{n-1}+\cdots+a_1x+a_0)^m=0$. Thus $x$ is integral over the prime subring of $T$ and we are done. For the second part of the statement of proposition, note that $T$ is integral over its center and $J(T)$ is nil, by \cite[Corollary 4.19]{lam}, thus we are done by the first part.
\end{proof}

Note that if $T$ is a ring with prime subring $\mathbf{Z}$, then one can easily see that $U(T)$ is integral over the prime subring of $T$ if and only if for each $x\in U(T)$ we have $x\in\mathbf{Z}[x^{-1}]$ if and only if for each subring $R$ of $T$ we have $U(R)=U(T)\cap R$.

\vspace{5mm}

By the proof of Theorem \ref{t19}, we also have the following result, which is a generalization of Corollary \ref{t10}.

\begin{cor}\label{t24a}
Let $T$ be a ring with $J(T)=0$, then either $T$ has a maximal subring of the center of $T$ is integrally closed in $T$.
\end{cor}

Consequently, if $T$ is a von Neumann regular ring which is integral over its center, we immediately conclude that either $T$ has a maximal subring or $T$ is a commutative ring.
In the next result we see that if the center of a $J$-semisimple (semiprimitive) ring $T$ has a finite minimal ring extension in $T$, then $T$ has a maximal subring.

\begin{cor}\label{t24b}
Let $T$ be a ring with $J(T)=0$. If the center of $T$ has a minimal extension of finite type in $T$, then $T$ has a maximal subring.
\end{cor}
\begin{proof}
Let $C$ be the center of $T$ and $R$ be a subring of $T$ which is minimal over $C$ of finite type (i.e., $R$ is integral over $C$, also note that clearly $R$ is commutative). Hence we are done by Corollary \ref{t24a}.
\end{proof}

Since the center of a simple ring is a field, we have the following immediate result from Corollary \ref{t24a}, which is a generalization of Corollary \ref{t10}.

\begin{cor}\label{t24c}
Let $T$ be a simple ring, then either $T$ has a maximal subring or the center of $T$ is algebraically closed in $T$.
\end{cor}

In item $(2)$ of the following fact we generalize Corollary \ref{t10}, for prime rings.

\begin{prop}\label{t24d}
Let $T$ be a ring with $J(T)=0$. Then the following hold:
\begin{enumerate}
\item If $T$ is torsionfree over its center, then either $T$ has a maximal subring or the center of $T$ is algebraically closed in $T$.
\item If $T$ is a prime ring, then either $T$ has a maximal subring or $T$ is a domain and the center of $T$ is algebraically closed in $T$.
\item If $T$ is prime and algebraic over its center, then either $T$ has a maximal subring or $T$ is an integral domain (with $|T|\leq 2^{2^{\aleph_0}}$).
\end{enumerate}
\end{prop}
\begin{proof}
$(1)$ Assume that $T$ has no maximal subring, then by Corollary \ref{t24a} it is clear that the center of $T$, say $C$, is integrally closed in $T$. Now assume that $x\in T$ is algebraic over $C$. Hence there exists $0\neq c\in C$ such that $cx$ is integral over $C$. Thus $cx\in C$. Now for each $y\in T$ we have $cxy=ycx$ and since $c\in C$ and $T$ is torsionfree over $C$, we immediately conclude that $xy=yx$. Hence $x\in C$. For $(2)$, first note that it easy to see that a prime ring is torsionfree over its center. Now assume that $T$ has no maximal subring, therefore by $(1)$ the center of $T$ is algebraically closed in $T$. Also note that since $J(T)=0$, by Corollary \ref{t2} we deduce that $T$ is a reduced ring. Therefore $T$ is a reduced prime ring and hence is a domain. $(3)$ is immediate by $(2)$ (and \cite[Corollary 3.7]{azkrmc}).
\end{proof}

If $T$ is a prime ring, then the center of $T$, say $C$, is an integral domain. Now assume that $T$ is integral over its center, if $C$ is a field, then by the previous proposition we deduce that either $T$ has a maximal subring or $T$ is an absolutely algebraic field without maximal subring. Hence, assume that $C$ is not a field. In the next result we prove that the ring $T_X$ (the ring of quotients of $T$ respect to the cental multiplicatively closed set $X:=C\setminus\{0\}$, see \cite{rvn}), has a maximal subring. Obviously, this result is a generalization of the fact that the quotient field of any non-field integral domain has a maximal subring, see \cite[Proposition 2.3]{azkrf}.

\begin{prop}\label{t24d1}
Let $T$ be a prime ring which is integral over its center $C$, where $C$ is not a field. Then $J(T_X)$ is a nil ideal and $T_X$ has a maximal subring, where $X:=C\setminus\{0\}$.
\end{prop}
\begin{proof}
Let $K$ be the quotient field of $C$, then it is not hard to see that $T_X$ is algebraic over $K$ and $K\subseteq C(T_X)$. Therefore by \cite[Corollary 4.19]{lam}, $J(T_X)$ is nil. For the final part note that $K$ is not an absolutely algebraic field, for $C$ is not a field. Thus $T_X$ has a maximal subring by Corollary \ref{t21} or Corollary \ref{t23}.
\end{proof}

\begin{rem}\label{t24e}
Let $\{A_i\ |\ i\in I\ \}$ be a family of ideals of a ring $T$. Then clearly for each $i\in I$, the center of $T/A_i$ is of the form $C_i/A_i$, where $C_i$ is a subring of $T$ which contains $A_i$. It is easy to see that $C(T)\subseteq \bigcap_{i\in I} C_i$. Conversely, if $\bigcap_{i\in I} A_i=0$, then it is not hard to see that $C(T)\supseteq \bigcap_{i\in I} C_i$ and therefore $C(T)=\bigcap_{i\in I} C_i$.
\end{rem}

\begin{prop}\label{t24g}
Let $T$ be a Hilbert semiprime/reduced ring. Then either $T$ has a maximal subring or the following hold:
\begin{enumerate}
\item For each prime ideal $P$ of $T$, $T/P$ is a domain and the center of $T/P$ is algebraically closed in $T/P$
\item The center of $T$ is integrally closed in $T$.
\end{enumerate}
\end{prop}
\begin{proof}
Assume that $T$ has no maximal subring. Let $P$ be a prime ideal of $T$, since $T$ is a Hilbert ring, we deduce that $J(T/P)=0$. Therefore by $(2)$ of Proposition \ref{t24d}, we conclude that $T/P$ is a domain and $C(T/P)$ is algebraically closed in $T/P$, for $T/P$ has no maximal subring. Thus $(1)$ holds. Now note that $Nil_*(T)=0$ and therefore $\bigcap_{P\in Spec(T)} P=0$. Hence by Remark \ref{t24e}, we deduce that $C(T)=\bigcap_{P\in Spec(T)} C_P$, where $C_P$ is a subring of $T$ for it we have $C(T/P)=C_P/P$. It is clear that if $x\in T$ is integral over $C(T)$, then for each prime ideal $P$ of $T$, $x+P$ is integral over $C(T/P)$ and therefore by $(1)$ we have $x+P\in C(T/P)=C_P/P$. Therefore $x\in C_P$, for each prime ideal $P$ of $T$ and hence $x\in C(T)$. Thus $(2)$ holds.
\end{proof}

\section{Artinian, Noetherian, Hilbert and direct product of rings}
Let $T$ be a ring which is integral over its center $C(T)$. In this section, we investigate about the existence of maximal subrings of $T$ whenever $T$ or $C(T)$ is Artinian or Noetherian and with a condition on cardinality of them. Also we consider rings which are direct product of an infinite family of rings (which are integral over their centers). First we begin by the following result.

\begin{prop}\label{ta24}
Let $T$ be a ring which is integral over its center $C:=C(T)$. If any of the following conditions holds, then $T$ has a maximal subring.
\begin{enumerate}
\item $C(T)$ is an Artinian ring which is either uncountable or of zero characteristic.
\item $C(T)$ is a Noetherian ring with $|R|>2^{\aleph_0}$.
\end{enumerate}
\end{prop}
\begin{proof}
For $(1)$, first note that by \cite[Propositions 1.4 and 2.4]{azkra}, $C$ has a maximal ideal $M$ such that $C/M$ is a field which is either uncountable or of zero characteristic. In particular in any cases, $C/M$ is not an absolutely algebraic field. Now, by \cite[Proposition 1.2]{blair}, $T$ has a prime ideal $Q$ such that $Q\cap C=M$. Clearly, $C/M\subseteq C(T/Q)$ and $T/Q$ is integral over $C/M$ and therefore $T/Q$ is integral over its center. Thus by Corollary \ref{t21}, we deduce that $T/Q$ has a maximal subring. Therefore $T$ has a maximal subring and we are done for $(1)$. Now we prove $(2)$, first assume that $C$ is an integral domain. By the proof of \cite[Corollary 2.7]{azkrc}, for each maximal ideal $M$ of $C$, there exists a natural number $n$, such that $C/M^n$ is an uncountable ring. Since $C/M^n$ is a zero dimensional noetherian local ring with only maximal ideal $M/M^n$, by the first part of the proof of $(1)$, we deduce that $C/M$ is an uncountable field and similar to the proof of $(1)$, we conclude that $T$ has a maximal subring. If $C$ is not an integral domain, then by the proof of \cite[Theorem 2.9]{azkrc}, $C$ has a minimal prime ideal $P$ such that $|C/P|=|C|$. By \cite[Proposition 1.2]{blair}, $T$ has a prime ideal $Q$ such that $Q\cap C=P$. Clearly $T/Q$ is integral over $C/P\subseteq C(T/Q)$ and therefore $T/Q$ is integral over its center. Similar to the first part of the proof of $(2)$, for each maximal ideal $M/P$ of $C/P$, the field $(C/P)/(M/P)$ is uncountable. Since $T/Q$ is integral over $C/P$, by \cite[Proposition 1.2]{blair}, $T/Q$ has a prime ideal $N/Q$ such that $(N/Q)\cap (C/P)=M/P$ and clearly $(T/Q)/(N/Q)$ is integral over $(C/P)/(M/P)\subseteq C((T/Q)/(N/Q))$, and therefore $(T/Q)/(N/Q)$ is integral over its center. Now since $(T/Q)/(N/Q)$ contains an uncountable field $(C/P)/(M/P)$, we immediately conclude that $(T/Q)/(N/Q)$ has a maximal subring by Corollary \ref{t21}. Thus $T$ has a maximal subring too.
\end{proof}

\begin{prop}\label{t25}
Let $T$ be a left noetherian algebraic $K$-algebra, where $K$ is a field. Then either $T$ has a maximal subring or $T$ is a left artinian ring. In particular, if $T$ has no maximal subring, then $T/J(T)$ is countable and $K$ is an absolutely algebraic field.
\end{prop}
\begin{proof}
Assume that $T$ has no maximal subring. Since $T$ is an algebraic $K$-algebra, by \cite[Theorem 4.20]{lam} we deduce that $J(T)$ is nil. Thus we conclude that $J(T)=Nil^*(T)$ is a nilpotent ideal of $T$, for $T$ is a left noetherian ring, see \cite[Theorem 10.30]{lam}. By $(1)$ of Proposition \ref{t20}, it is clear that $T/J(T)$ is a commutative ring. Since $T/J(T)$ is algebraic over $K$, we immediately conclude that $T/J(T)$ is a zero-dimensional commutative noetherian ring and therefore is an artinian ring. Now since $J(T)$ is a nilpotent ideal of $T$, we immediately conclude that $R$ is a left artinian ring, see \cite[Theorem 4.15]{lam}.
\end{proof}

Although the next lemma is well known but we prove it for the sake of completeness.

\begin{lem}\label{t26}
Let $R\subseteq T$ be a ring extension and $R$ is a left Artinian ring. Then $U(R)=U(T)\cap R$ and $J(T)\cap R\subseteq J(R)$.
\end{lem}
\begin{proof}
Let $x\in U(T)\cap R$, then clearly $Rx\cong R$ as left $R$-modules, for $x\in U(T)$ and therefore $l.ann_R(x)=0$. Thus $Rx$ as a left $R$-module has finite length which is equal to the length of ${}_RR$. Since the length of ${}_RR$ is finite and $Rx\leq R$, we conclude that $Rx=R$. This immediately shows that $x^{-1}\in R$. Thus $U(R)=U(T)\cap R$. For the final part, assume that $a\in J(T)\cap R$ and $b\in R$. Thus $ab\in J(T)$ and therefore $1-ab\in U(T)$. Thus $1-ab\in U(T)\cap R=U(R)$, by the first part. Thus $a\in J(R)$ and we are done.
\end{proof}

Now the following is in order.

\begin{thm}\label{t27}
Let $R$ be a maximal subring of a ring $T$, and $R,\ T$ are one-sided Artinian rings. Then either $(R:T)\in Max(T)$ or $J(R)=R\cap J(T)$.
\end{thm}
\begin{proof}
First note that each left primitive ideal $Q$ of $T$ is a maximal ideal of $T$, for $Q$ is prime and $T$ is a zero-dimensional ring (note $T$ is a one-sided Artinian ring). Hence if $R$ contains a left primitive ideal $Q$ of $T$, then $Q\subseteq (R:T)$ and therefore $(R:T)=Q$ is a maximal ideal of $T$. Thus assume that $R$ does not contain any left primitive ideal of $T$. For each left primitive ideal $P$ of $T$, we deduce that $R+P=T$, for $R$ is a maximal subring of $T$. Hence we have the rings isomorphism $R/(R\cap P)\cong (R+P)/P=T/P$, which means $R\cap P$ is a maximal ideal of $R$ and therefore $J(R)\subseteq R\cap P$. Thus $J(R)\subseteq J(T)\cap R$ and therefore by the previous lemma we infer that $J(R)=J(T)\cap R$.
\end{proof}

Note that in the previous theorem in fact we may assume that $T$ is a zero-dimensional ring and $R$ is a one-sided Artinian ring.

\begin{prop}\label{t27a}
Let $T$ be a zero-dimensional ring which is integral over its center. Then either $T$ has a maximal subring or each subring $R$ of $T$ is zero-dimensional. In particular, if $T$ has no maximal subring, then $Char(T)=n>0$ and $T$ is integral over $\mathbb{Z}_n$.
\end{prop}
\begin{proof}
Assume that $T$ has no maximal subring. Thus by $(2)$ of Proposition \ref{t20}, for each maximal ideal $M$ of $T$, the field $T/M$ is an absolutely algebraic field, hence each subring of $T/M$ is a field. Thus for each subring $S$ of $T$, since $(S+M)/M$ is a subring of $T/M$ and we have the rings isomorphism $S/(S\cap M)\cong (S+M)/M$, we deduce that $S\cap M$ is a maximal ideal of $S$. Now assume that $P$ is a prime ideal of $S$, then $X=S\setminus P$ is an $m$-system in $T$, thus $T$ has a prime ideal $Q$ such that $Q\cap X=\emptyset$. Therefore $Q\cap S\subseteq P$. Since $T$ is a zero-dimensional ring, we conclude that $Q$ is a maximal ideal of $T$ and therefore $Q\cap S$ is a maximal ideal of $S$, which by the first part implies that $P=Q\cap S$ is a maximal ideal of $S$, i.e., $S$ is a zero-dimensional ring. Hence the first part holds. Since each subring of $T$ is a zero-dimensional ring and $dim(\mathbb{Z})=1$, we deduce that $Char(T)=n>0$. Now for each $x\in T$, the commutative ring $R=\mathbb{Z}_n[x]$ is a subring of $T$ and each subring of $R$ is zero-dimensional, for each subring of $R$ is a subring of $T$. Thus by \cite[Theorem 1.3]{gil}, we infer that $x$ is integral over $\mathbb{Z}_n$. Hence $T$ is integral over $\mathbb{Z}_n$.
\end{proof}

A ring $T$ is called hereditary zero dimensional, if each subring of $T$ is zero-dimensional. Hence by the previous theorem if $T$ is a zero-dimensional ring which is integral over its center, and either $Char(T)=0$ or $T$ is not integral over its prime subring, then $T$ has a maximal subring. Also note that if $T$ is a ring with $Char(T)=n>0$ and $T$ is integral over $\mathbb{Z}_n$, then $T$ is a periodic ring, i.e., for each $x\in T$, there exist natural numbers $n\neq m$ such that $x^n=x^m$. To see this, since $x$ is integral over $\mathbb{Z}_n$, we immediately conclude that $\mathbb{Z}_n[x]$ is finite and note that $\{x^k\ |\ k\in\mathbb{N}\}\subseteq \mathbb{Z}_n[x]$. Hence $\{x^k\ |\ k\in\mathbb{N}\}$ is finite and therefore we are done.  In the next result we generalize \cite[Proposition 2.4 and Corollary 2.5]{azkra}. We remind the reader that if $T$ is a countable ring and $M$ is a finitely generated $R$-module, then there exists a natural number $n$, such that $M$ is a homomorphic image of $T^n$, as left $T$-module. Therefore $M$ is countable too. Now the following is in order.

\begin{thm}\label{t28}
Let $T$ be a left Artinian ring which is integral over its center. Then either $T$ has a maximal subring or $T$ is countable with nonzero characteristic and integral over its prime subring.
\end{thm}
\begin{proof}
Assume that $T$ has no maximal subring. Thus by $(1)$ of Proposition \ref{t20}, we conclude that $T/J(T)$ is a commutative ring. Since $T$ has no maximal subring, we infer that $T/J(T)$ has no maximal subring too and therefore by \cite[Corollary 2.4]{azn}, $T/J(T)$ is a countable ring. Now we claim that $T$ is countable too. First note that there exists a natural number $m$ such that $J(T)^m=0$, for $T$ is a left Artinian ring. Now consider the chain $0=J(T)^m\subseteq J(T)^{m-1}\subseteq\cdots\subseteq J(T)\subseteq T$. For each $i$, the left $T$-module $J(T)^i/J(T)^{i+1}$, $0\leq i\leq m-1$ is a finitely generated $T/J(T)$-module (note, $T$ is a left noetherian ring and therefore each left ideal of $T$ is finitely generated). Thus $J(T)^i/J(T)^{i+1}$ is countable. In particular $J(T)^{m-1}$ is countable. Therefore for each $i$ we deduce that $J(T)^i$ is countable and thus $T$ is countable too. The final part is evident by Proposition \ref{t27a}.
\end{proof}

Now we want to prove that if a left Noetherian ring $T$ is integral over its center and has no maximal subring, then $|T|\leq 2^{\aleph_0}$.  We remind that if $T$ is a Noetherian commutative ring which have no maximal subring, then $T$ is countable, see \cite[Corollary 2.4]{azn}. We need the following lemma.

\begin{lem}\label{t31}
Let $T$ be a left noetherian ring which is integral over center and $J(T)=Nil_*(T)$ (i.e., $J(T)$ is nil). Then either $T$ has a maximal subring or $T$ is countable.
\end{lem}
\begin{proof}
Assume that $T$ has no maximal subring. Thus by $(1)$ of Proposition \ref{t20}, we conclude that $T/J(T)$ is a commutative ring. Since $T$ has no maximal subring, then we infer that $T/J(T)$ has no maximal subring too and therefore by \cite[Corollary 2.4]{azn}, $T/J(T)$ is a countable ring. Since $T$ is a left Noetherian ring and $J(T)$ is a nil ideal, we deduce that $J(T)$ is nilpotent, see \cite[Theorem 10.30]{lam}. Hence by a similar proof of Theorem \ref{t28}, we deduce that $T$ is countable.
\end{proof}

Now we have the following main result.

\begin{thm}\label{t32}
Let $T$ be a left Noetherian ring which is integral over its center. Then either $T$ has a maximal subring or $|T|\leq 2^{\aleph_0}$.
\end{thm}
\begin{proof}
Assume that $T$ has no maximal subring. By \cite[Theorem 2.4]{blair}, we have $\bigcap_{n=1}^\infty (J(T))^n=0$. Thus $T$ embeds in $\prod_{n=1}^\infty T/J(T)^n$. Now for each natural number $n$, $T/J(T)^n$ is a left Noetherian ring which is integral over its center and clearly $J(T/J(T)^n)=J(T)/J(T)^n$ is nilpotent. Since $T$ has no maximal subring, we infer that $T/J(T)^n$ has no maximal subring for each $n$. Thus by Lemma \ref{t31}, we deduce that $T/J(T)^n$ is countable for each $n$. Therefore
$$|T|\leq \prod_{n=1}^\infty |T/J(T)^n|\leq 2^{\aleph_0}.$$
\end{proof}

In the next observations we put some conditions on the center of a ring, for example we assume that the center is a Hilbert ring (especially, if the center is a zero-dimensional ring). First we need the following lemma.

\begin{lem}\label{t33}
Let $T$ be a domain which is integral over its center. If $J(C(T))=0$, then either $T$ has a maximal subring or $T$ is a Hilbert integral domain with $J(T)=0$ and $|T|\leq 2^{2^{\aleph_0}}$.
\end{lem}
\begin{proof}
Assume that $T$ has no maximal subring. We claim that $J(T)=0$. Assume that $0\neq x\in J(T)$. Since $x$ is integral over $C(T)$ we infer that there exist $n$ and $c_0,c_1,\ldots,c_{n-1}\in C(T)$ such that $x^n+c_{n-1}x^{n-1}+\cdots+c_1x+c_0=0$, and $c_0\neq 0$ for $T$ is a domain. Hence $c_0\in J(T)\cap C(T)$. By \cite[Corollary 1.4]{blair}, $J(T)\cap C(T)=J(C(T))$ and therefore $c_0=0$ which is a contradiction. Thus $J(T)=0$ and therefore we are done by $(1)$ of Proposition \ref{t20}.
\end{proof}

\begin{cor}\label{t34}
Let $T$ be a ring which is integral over its center. If $C(T)$ is a Hilbert ring, then either $T$ has a maximal subring or for each completely prime ideal $Q$ of $T$, $T/Q$ is a Hilbert integral domain and therefore $Q$ is an intersection of a family of maximal (one-sided) ideal of $T$ (i.e., $J(T/Q)=0$).
\end{cor}
\begin{proof}
Assume that $T$ and therefore $T/Q$ have no maximal subring. First note that one can easily see that $Q\cap C(T)$ is a prime ideal of $C(T)$ and $T/Q$ is integral over the integral domain $A:=(C(T)+Q)/Q\subseteq C(T/Q)$. Since $C(T)$ is a Hilbert ring, we conclude that $J(C(T)/(C(T)\cap Q))=0$ and therefore $J(A)=0$ too. Now similar to the proof of the previous lemma we deduce that $T/Q$ is commutative Hilbert ring (note that by \cite[Corollary 1.4]{blair}, $J(A)=J(T/Q)\cap A$) and hence we are done.
\end{proof}

\begin{prop}\label{t35}
Let $T$ be a ring which is integral over its center. If $C(T)$ is a zero dimensional ring, then each strongly prime ideal $Q$ of $T$, is an intersection of a family of maximal left/right ideals of $T$ (i.e., $J(T/Q)=0$). In particular, $J(T)=Nil^*(T)$. Moreover, if $T$ has no maximal subring, then $T/Q$ is a field for each strongly prime ideal $Q$ of $T$, and $U(T)$ is integral over the prime subring of $T$.
\end{prop}
\begin{proof}
First note that $Q\cap C(T)$ is a prime ideal of $C(T)$ and therefore by our assumption is a maximal ideal of $C(T)$. Hence $E:=(C(T)+Q)/Q$ is a field and clearly $T/Q$ is integral over $E$. Hence by \cite[Corollary 4.19]{lam}, we conclude that $J(T/Q)$ is a nil ideal of $T/Q$. Thus $J(T/Q)=0$, for $T/Q$ is a strongly prime ring. Therefore $Q$ is an intersection of a family of maximal left/right ideals of $T$. Hence the first part holds. For the next part note that it is well-known that $Nil^*(T)\subseteq J(T)$, and $Nil^*(T)$ is the intersection of all strongly prime ideals of $T$, see \cite[Proposition 2.6.7]{rvn}. Now since each strongly prime ideal of $T$ is an intersection of a family of maximal left/right ideals of $T$, we deduce that $J(T)\subseteq Nil^*(T)$ and therefore the equality holds. Now assume that $T$ has no maximal subring, thus $T/Q$ has no maximal subring too and therefore by $(1)$ of Proposition \ref{t20}, we conclude that $T/Q$ is a Hilbert integral domain. Now note that by the first part of the proof $T/Q$ is integral over the field $E$, hence $T/Q$ is a field. The final part is evident by Proposition \ref{t24a1}.
\end{proof}

Now we have the following main result for reduced rings.

\begin{thm}\label{t36}
Let $T$ be a reduced ring which is integral over its center. If $C(T)$ is a Hilbert ring (in particular, if $C(T)$ is a zero-dimensional ring), then either $T$ has a maximal subring or $T$ is a commutative Hilbert (von Neumann regular) ring with $|T|\leq 2^{2^{\aleph_0}}$.
\end{thm}
\begin{proof}
First note that if $P$ is a minimal prime ideal of $T$, then $T/P$ is a domain by \cite[Lemma 12.6]{lam}, i.e., $P$ is a completely prime ideal. Now assume that $T$ has no maximal subring and $P$ be any minimal prime ideal of $T$. Since $T$ has no maximal subring, we conclude that $T/P$ is a Hilbert integral domain by Proposition \ref{t34}. This immediately implies that $T$ is a Hilbert ring, for each prime ideal of $T$ contains a minimal prime ideal of $T$. In other words, for each prime ideal $P$ of $T$, $T/P$ is a Hilbert integral domain and therefore $J(T/P)=0$, i.e., $P$ is an intersection of a family of maximal (left/right) ideals of $T$. Thus $J(T)=Nil^*(T)=0$ and therefore by Theorem \ref{t19}, $T$ is a commutative ring. Hence by $(1)$ of Theorem \ref{t20}, we deduce that $|T|\leq 2^{2^{\aleph_0}}$. Now, if $C(T)$ is a zero-dimensional ring, then since each completely prime ideal is strongly prime ideal, thus by Proposition \ref{t35}, we conclude that each minimal prime ideal of $T$ is maximal and therefore $T$ is a zero-dimensional commutative ring. Hence $T$ is a von Neumann regular ring, for $T$ is a reduced ring.
\end{proof}

\begin{rem}
Let $T$ be a reduce ring which is integral over its center. If $C(T)$ is a zero-dimensional ring, then $T$ is a a strongly von Neumann regular ring (and therefore is a duo ring). To see this note that, since $T$ is a reduced ring, we must show that $T$ is a von Neumann regular ring. Hence by \cite[Theorem 1.21]{vnrg}, it suffices to show that for each completely prime ideal $P$ of $T$, the ring $T/P$ is a von Neumann regular. Let $P$ be a completely prime ideal of $T$, i.e., $T/P$ is a domain. It is not hard to see that $P\cap C(T)$ is a prime ideal of $C(T)$, and therefore by the assumption is a maximal ideal of $C(T)$. Therefore $C(T)/(C(T)\cap P)$ is a field. Hence $E:=(C(T)+P)/P\cong C(T)/(C(T)\cap P)$ is a field. Now note that by the assumption the domain $D:=T/P$ is integral over $E\subseteq C(D)$. Similar to the proof of commutative rings, this immediately implies that $D$ is a division ring and therefore is a von Neumann regular ring. Hence we are done.
\end{rem}

\begin{cor}\label{t38}
Let $T$ be a left Noetherian reduced ring which is integral over its center. If $C(T)$ is a Hilbert ring (in particular, if $C(T)$ is a zero-dimensional ring), then either $T$ has a maximal subring or $T$ is a countable commutative Noetherian (Artinian) ring. In particular, if $T$ is a left noetherian reduced ring which is integral over its center, $C(T)$ is a Hilbert ring (in particular, whenever $C(T)$ is a zero-dimensional ring) and $T$ is not right noetherian, then $T$ has a maximal subring.
\end{cor}
\begin{proof}
If $T$ has no maximal subring, then by the previous theorem, $T$ is a commutative ring. Therefore by \cite[Corollary 2.4]{azn}, $T$ is a countable ring. If $C(T)$ is a zero-dimensional ring, then by the previous theorem $R$ is a commutative zero-dimensional ring and therefore $T$ is an Artinian ring. The final part is evident.
\end{proof}

\begin{rem}\label{t39}
Let $T$ be a ring which is integral over its center. If $T$ has no maximal subring and $T$ is Hilbert, then $C(T)$ is Hilbert too. To see this, first note that for each maximal ideal $M$ of $T$, the field $T/M$ is an absolutely algebraic by $(2)$ of Proposition \ref{t20}. Thus each subring of $T/M$ is a field. Hence for each subring $R$ of $T$, we conclude that $R/(R\cap M)$ is an absolutely algebraic field, for $(R+M)/M$ is a subfield of $T/M$. Now, if $P$ is a prime ideal of $C(T)$, then by \cite[Proposition 1.2]{blair}, $T$ has a prime ideal $Q$ such that $Q\cap C(T)=P$. Since $T$ is a Hilbert ring, there exist maximal ideals $M_{\alpha}$, $\alpha\in \Gamma$, such that $Q=\bigcap_{\alpha\in\Gamma} M_{\alpha}$. Therefore $P=Q\cap C(T)=\bigcap_{\alpha\in\Gamma} (M_{\alpha}\cap C(T))$, and by the first part for each $\alpha\in\Gamma$, $M_{\alpha}\cap C(T)$ is a maximal ideal of $C(T)$. Therefore $C(T)$ is a Hilbert ring.
\end{rem}

If $T$ is a commutative ring which has no maximal subring, then for each subring $R$ of $T$, by \cite[Theorem 2.4]{azconch}, $U(R)=U(T)\cap R$; and by \cite[Proposition 2.3]{azalicid}, $J(R)=J(T)\cap R$. Now we have the following result.

\begin{prop}\label{t40}
Let $T$ be a ring which is integral over its center and $R$ be a subring of $T$ with $J(T)\cap R\subseteq R$. Then either $T$ has a maximal subring or $J(R)=J(T)\cap R$ and $U(R)=U(T)\cap R$.
\end{prop}
\begin{proof}
Assume that $T$ has no maximal subring, thus by Corollary \ref{t22}, for each maximal (left) ideal $M$ of $T$, we conclude that $R\cap M$ is a maximal (left) ideal of $R$ and therefore $J(R)\subseteq J(T)\cap R$. Hence by the assumption of the theorem we conclude that $J(R)=J(T)\cap R$. It is clear that $U(R)\subseteq U(T)\cap R$. Let $x\in U(T)\cap R$, the $x^{-1}+J(T)$ is a unit in $T/J(T)$. Since $T$ has no maximal subring, by $(1)$ of Proposition \ref{t20}, we deduce that $T/J(T)$ is a commutative ring and clearly $T/J(T)$ has no maximal subring too. Thus by \cite[Corollary 2.5]{azconch}, the units group of the ring $T/J(T)$ is integral over the prime subring of $T/J(T)$. Therefore, there exist $n\in\mathbb{N}$ and $a_0,\ldots,a_{n-1}$ in the prime subring of $T$ such that
$$(x^{-1}+J(T))^n+a_{n-1}(x^{-1}+J(T))^{n-1}+\cdots+a_1(x^{-1}+J(T))+a_0=0\ \text{in the ring}\ T/J(T).$$
Hence $(x^{-1})^n+a_{n-1}(x^{-1})^{n-1}+\cdots+a_1(x^{-1})+a_0\in J(T)$. Multiplying by $x^n$, we obtain that $1+a_{n-1}x+\cdots+a_1x^{n-1}+a_nx^n\in J(T)$. Now note that $x\in R$ and since $R$ is a subring of $T$, we deduce that $R$ contains the prime subring of $T$, i.e., $a_i\in R$. Thus $1+a_{n-1}x+\cdots+a_1x^{n-1}+a_nx^n\in J(T)\cap R$. Hence by the first part, we conclude that $1+a_{n-1}x+\cdots+a_1x^{n-1}+a_nx^n\in J(R)$. Therefore $a_{n-1}x+\cdots+a_1x^{n-1}+a_nx^n\in U(R)$ which immediately implies that $x\in U(R)$. Hence the equality holds.
\end{proof}

\begin{prop}\label{t40a}
Suppose for each maximal ideal $M$ of a ring $T$ there exists a prime number $p$ such that $C(T/M)$ is algebraic over $\mathbb{F}_p$. Then either $T$ has a maximal subring or $J(C(T))=C(T)\cap J(T)$.
\end{prop}
\begin{proof}
Let $C:=C(T)$, it is clear that $U(C)=C\cap U(T)$ and therefore $J(T)\cap C\subseteq J(C)$. Now assume that $T$ has no maximal subring. Therefore by $(1)$ of Theorem \ref{t1}, $T$ is a quasi duo ring. Now for each maximal (left/right) ideal $M$ of $T$, we have $C/(C\cap M)\cong (C+M)/M\leq C(T/M)$. Therefore $C\cap M$ is a maximal ideal of $C$ and thus $J(C)\subseteq M$. Hence $J(C)\subseteq J(T)$, which immediately implies that $J(C)\subseteq C\cap J(T)$ and therefore the equality holds.
\end{proof}

Finally in this paper we have two results about maximal subring in infinite direct product of rings.

\begin{thm}\label{t41}
Let $\{T_i\}_{i\in I}$ be a family of rings, $I$ is infinite and $T=\prod_{i\in I} T_i$. Then the following hold:
\begin{enumerate}
\item If each $T_i$ is integral over its center, then $T$ has a maximal subring.
\item If $T$ is integral over its center, then $T$ has a maximal subring.
\end{enumerate}
\end{thm}
\begin{proof}
$(1)$ If for some $i\in I$, the ring $T_i$ has a maximal subring, then clearly $T$ has a maximal subring. Thus assume that for each $i\in I$, $T_i$ has no maximal subring. Therefore by $(1)$ of Proposition \ref{t20}, for each $i\in I$, the ring $T/J(T_i)$ is a commutative ring. Take $J:=\prod_{i\in I} J(T_i)$, then clearly $T/J\cong \prod_{i\in I} T/J(T_i)$. Thus $T/J$ is a product of an infinite family of commutative rings, and therefore by \cite[Theorem 3.17]{azkrc}, $T/J$ has a maximal subring and so does $T$. Hence $(1)$ holds. For $(2)$, it suffices to show that each $T_i$ is integral over its center and use $(1)$. Clearly, $C_T(T)=\prod_{i\in I}C_{T_i}(T_i)$. Let $k\in I$ and $t\in T_k$, and put $x=(x_i)_{i\in I}$ where $x_i=0$ for $i\neq k$ and $x_k=t$. By our assumption, $x$ is integral over the center of $T$. Hence there exist $n\in\mathbb{N}$ and $c_0, c_1,\ldots, c_{n-1}$ in $C_T(T)$ such that $x^n+c_{n-1}x^{n-1}+\cdots+c_1x+c_0=0$. Since $C_T(T)=\prod_{i\in I}C_{T_i}(T_i)$, we have $c_j=(c_{ji})_{i\in I}$, for $0\leq j\leq n-1$, where $c_{ji}\in C_{T_i}(T_i)$. Therefore by calculating the $k$-th component from $x^n+c_{n-1}x^{n-1}+\cdots+c_1x+c_0=0$, we immediately conclude that $t^n+c_{n-1k}t^{n-1}+\cdots+c_{1k}t+c_{0k}=0$. Hence $t$ is integral over $C_{T_i}(T_i)$ and we are done by $(1)$.
\end{proof}

We conclude this paper by our final result about the structure of maximal subrings in infinite direct product of rings.

\begin{prop}\label{t42}
Let $\{T_i\}_{i\in I}$ be a family of rings, $I\neq \emptyset$ and $T=\prod_{i\in I} T_i$. If $R$ is a maximal subring of $T$, then one of the following holds:
\begin{enumerate}
\item $R=\prod_{i\in I} R_i$, where there exists a unique $j\in I$ such that $R_j$ is a maximal subring of $T_j$ and for each $i\in I\setminus\{j\}$, we have $R_i=T_i$.
\item There exists a family $\{A_i\}_{i\in I}$ of ideals of $R$ such that $\bigcap_{i\in I}A_i=0$ and $T_i\cong R/A_i$, for each $i\in I$.
\end{enumerate}
\end{prop}
\begin{proof}
For each $j\in I$, let $B_j=\prod_{i\in I} K_i$, where $K_i=T_i$ for $i\in I\setminus\{j\}$ and $K_j=\{0\}$. It is clear that $B_i$ is an ideal of $T$ and $T/B_i\cong T_i$. Now if there exists $j\in I$, such that $B_j\subseteq R$, then clearly $(1)$ holds. Hence assume that for each $i\in I$, $B_i$ is not contained in $R$. Then by maximality of $R$, we conclude that $R+B_i=T$ and therefore $R/(R\cap B_i)\cong (R+B_i)/B_i=T/B_i\cong T_i$. Hence if we put $A_i:=R\cap B_i$, the we conclude that for each $i\in I$ we have $R/A_i\cong T_i$. Finally note that $\bigcap_{i\in I}B_i=0$ and therefore $\bigcap_{i\in I}A_i=0$. Thus $(2)$ holds.
\end{proof}

\vspace{0.5cm}
\centerline{\Large{\bf Acknowledgement}}
The author is grateful to the Research Council of Shahid Chamran University of Ahvaz (Ahvaz-Iran) for
financial support (Grant Number: SCU.MM1403.721)
\vspace{0.5cm}


\begin{thebibliography}{mm}
\bibitem{azalicid}
M. Alinaghizadeh, A. Azarang, Maximal subrings of classical integral domains, Quaest. Math. {\bf 46} (7) (2023) 1253–1272.


\bibitem{azq}
A. Azarang, Non-quasi duo rings have maximal subrings, Int. Math. Forum, {\bf 5} (20) (2010) 979-994.

\bibitem{azsid}
A. Azarang, Submaximal integral domains, Taiwanese J. Math., {\bf 17} (4) (2013) 1395-1412.


\bibitem{azn}
A. Azarang, On the existence of maximal subrings in commutative Noetherian rings, J. Algebra Appl. {\bf 14} (1) (2015) ID:1450073.



\bibitem{azconch}
A. Azarang, Conch maximal subrings, Comm. Algebra. {\bf 50} (3) (2022) 1267-1282.

\bibitem{azcond}
A. Azarang, The conductor ideals of maximal subrings in non-commutative rings, preprint (submitted) (2024) Arxiv: https://arxiv.org/abs/2406.12890.

\bibitem{azid}
A. Azarang, Classical ideals theory of maximal subrings in non-commutative rings, preprint (submitted) (2024) Arxiv: https://arxiv.org/abs/2406.12891.

\bibitem{azdiv}
A. Azarang, Maximal subrings of division rings, preprint (submitted) (2024) Arxiv:


\bibitem{azkra}
A. Azarang, O.A.S. Karamzadeh, On the existence of maximal subrings in commutative Artinian rings, J. Algebra Appl. {\bf 9} (5) (2010) 771-778.

\bibitem{azkrf}
A. Azarang, O.A.S. Karamzadeh, Which fields have no maximal subrings?, Rend. Sem. Mat. Univ. Padova, {\bf 126} (2011) 213-228.

\bibitem{azkrc}
A. Azarang, O.A.S. Karamzadeh, On maximal subrings of commutative rings, Algebra Colloq. {\bf 19} (Spec 1) (2012) 1125-1138.

\bibitem{azkrmc}
A. Azarang, O.A.S. Karamzadeh, Most commutative rings have maximal subrings, Algebra Colloq. {\bf 19} (Spec 1) (2012) 1139-1154.



\bibitem{blair}
W.D. Blair, Right Noetherian rings integral over their centers, J. Algebra {\bf 27} (1973) 187–198.





\bibitem{cohnfir}
P.M. Cohn, {\it Free ideal rings and localization in general rings}, New Math. Monogr. Cambridge University Press, Cambridge, (2006).

\bibitem{adjex}
L.I. Dechene, Adjacent extensions of rings, Ph.D. Dissertation, University of California, Riverside, (1978).

\bibitem{dbsid}
D.E. Dobbs, J. Shapiro, A classification of the minimal ring extensions of an integral domain, J. Algebra {\bf 305} (2006) 185-193.

\bibitem{dbsc}
D.E. Dobbs, J. Shapiro, A classification of the minimal ring extensions of certain commutative rings, J. Algebra {\bf 308} (2007) 800-821.


\bibitem{dorsy}
T.J. Dorsey, Z. Mesyan, On minimal extension of rings, Comm. Algebra, {\bf 37} (10) (2009) 3463-3486.


\bibitem{faith}
C. Faith, {\it Rings and things and a fine array of twentieth century associative algebra}, Mathematical Surveys and Monographs. 65. Providence, RI: American Mathematical Society (AMS). (1998).

\bibitem{frd}
D. Ferrand, J.-P. Olivier, Homomorphismes minimaux d’anneaux, J. Algebra {\bf 16} (1970) 461-471.

\bibitem{gil}
R. Gilmer, W.J. Heinzer, Products of commutative rings and zero-dimensionality, Trans. Amer. Math. Soc. {\bf 331} (1992) 663-680.

\bibitem{vnrg}
K.R. Goodearl, {\it von Neumann regular rings}, Monogr. Stud. Math., Pitman (Advanced Publishing Program), Boston, Mass.-London (1979).

\bibitem{good}
K.R. Goodearl, R.B. Warfield, JR, {\it An introduction to noncommutative Noetherian rings}, Camberidge University Press, Second Edition, (2004).

\bibitem{hazw}
M. Hazewinkel, N. Gubareni, V.V. Kirichenko, {\it Algebras, rings and modules}, Vol. 1. Math. Appl., 575 Kluwer Academic Publishers, Dordrecht, (2004).



\bibitem{klein}
A.A. Klein, The finiteness of a ring with a finite maximal subrings. Comm. Algebra {\bf 21} (4) (1993) 1389-1392.

\bibitem{laffey}
T.J. Laffey, A finiteness theorem for rings, Proc. R. Ir. Acad. {\bf 92} (2) (1992) 285-288.

\bibitem{lam}
T.Y. Lam, {\it A first course in noncmmutative rings}, Second Edition, Springer-Verlag, (2001).

\bibitem{lam2}
T.Y. Lam, {\it Lectures on Modules and Rings}, Springer-Verlag, (1999).


\bibitem{lamq}
T.Y. Lam, A.S. Dugas, Quasi-duo rings and stable range descent, J. Pure Appl. Algebra {\bf 195} (2005) 243-259.

\bibitem{lee}
T.K. Lee, K.S. Liu, Algebra with a finite-dimensional maximal subalgebra, Comm. Algebra {\bf 33} (1) (2005) 339-342.



\bibitem{marbo}
H. Marubayashi, H. Miyamoto, A. Ueda, {\it Non-commutative valuation rings and semi-hereditary orders},
K-Monogr. Math., Kluwer Academic Publishers, Dordrecht, (1997).




\bibitem{rvn}
L. Rowen, {\it Ring theory}, Vol. I, Pure Appl. Math., 127 Academic Press, Inc., Boston, MA, (1988).

\bibitem{abmin}
G. Picavet and M. Picavet-L'Hermitte. About minimal morphisms, in: Multiplicative ideal theory in commutative algebra, Springer-Verlag, New York, (2006) pp. 369-386.



\end{thebibliography}

\end{document}